\documentclass[11pt,a4paper]{article}%
\usepackage{amssymb}
\usepackage{amsfonts}
\usepackage{amsmath}
\usepackage{verbatim}
\usepackage{graphicx}%
\newtheorem{theorem}{Theorem}

\newtheorem{corollary}[theorem]{Corollary}

\newtheorem{lemma}[theorem]{Lemma}

\newtheorem{proposition}[theorem]{Proposition}

\newenvironment{proof}[1][Proof]{\noindent\textbf{#1.} }{\ \rule{0.5em}{0.5em}}
\begin{document}

\title{Complex conference matrices, complex Hadamard matrices and complex equiangular tight frames}
\author{Boumediene Et-Taoui\\Universit\'{e} de Haute Alsace -- LMIA\\4 rue des fr\`{e}res Lumi\`{e}re 68093 Mulhouse Cedex}
\maketitle

\begin{abstract}
In this article we construct new, previously unknown parametric families of complex conference
matrices and of complex Hadamard matrices of square orders and related them to
complex equiangular tight frames. It is shown that for any odd integer $k\geq 3$ such that $2k=p^{\alpha}+1$, $p$ prime, $\alpha$ non-negative integer, on the one hand there exists a $(2k,k)$ complex equiangular tight frame and for any $\beta\in\mathbb{N}^{\ast}$ there exists a $((2k)^{2^{\beta}},\frac{1}{2}(2k)^{2^{\beta-1}}((2k)^{2^{\beta-1}%
}\pm1))$ complex equiangular tight frame depending on one unit complex number, and on the other hand there exist a family of
$((4k)^{2^{\beta}},\frac{1}{2}(4k)^{2^{\beta-1}}((4k)^{2^{\beta-1}}\pm1))$ complex equiangular tight frames depending on two unit complex numbers.


\end{abstract}

B.Ettaoui@uha.fr \bigskip

\section{Introduction}

This article deals with the so-called notion of complex equiangular tight
frames. Important for coding and quantum information theories are real and
complex equiangular tight frames, for a survey we refer to \cite{HP}. In a
Hilbert space $\mathcal{H}$, a subset $F=\left\{  f_{i}\right\}  _{i\in
I}\subset\mathcal{H}$ is called a \textit{frame} for $\mathcal{H}$ provided
that there are two constants $C,D\ >0$ such that
\[
C\left\Vert x\right\Vert ^{2}\leq\sum\limits_{i\in I}\left\vert \left\langle
x,f_{i}\right\rangle \right\vert ^{2}\leq D\left\Vert x\right\Vert ^{2},
\]
holds for every $x\in\mathcal{H}$. If $C=D=1$ then the frame is called
\textit{normalized tight} or a \textit{Parseval frame}. Throughout this
article we use the term $(n,k)$ frame to refer to a Parseval frame of $n$
vectors in $\mathbb{C}^{k}$ equipped with the usual inner product. The ratio
$\frac{n}{k}$ is called the \textit{redundancy} of the $(n,k)$ frame. It is
well- known that any Parseval frame induces an isometric embedding of
$\mathbb{C}^{k}$ into $\mathbb{C}^{n}$ which maps $x\in\mathbb{C}^{k}$ to its
frame coefficients $(Vx)_{j}=\left\langle x,f_{j}\right\rangle $, called the
analysis operator of the frame. Because $V$ is linear, we may identify $V$
with an $n\times k$ matrix and the vectors $\left\{  f_{1},...,f_{n}\right\}
$ denote the columns of $V^{\ast}$, the Hermitean conjugate of $V$. It is
shown in \cite{HP} that a $(n,k)$ frame is determined up to a unitary
equivalence by its Gram matrix $VV^{\ast}$, which is a self-adjoint projection
of rank $k$. If in addition the frame is uniform and equiangular that is,
$\left\Vert f_{i}\right\Vert ^{2}$ and $\left\vert \left\langle f_{i}%
,f_{j}\right\rangle \right\vert $ are constants for all $1\leq i\leq n$ and
for all $i\neq j$, $1\leq i,j\leq n$, respectively, then follows
\[
VV^{\ast}=\frac{k}{n}I_{n}+\sqrt{\frac{k(n-k)}{n^{2}(n-1)}}Q,
\]
where $Q$ is a self-adjoint matrix with diagonal entries all $0$ and
off-diagonal entries all of modulus $1$, and $I_{n}$ is the identity matrix of
order $n$. The matrix $Q$ is called the \textit{Seidel matrix} or
\textit{signature matrix} associated with the $(n,k)$ frame. The existence of
an equiangular Parseval frame is known from \cite{HP} to be equivalent to the
existence of a Seidel matrix with two eigenvalues. Of course the
vectors of a $(n,k)$ frame generate equiangular lines in $\mathbb{C}^{k}$. However equiangular $n$-tuples are not characterized by
single matrices but by the classes of such matrices under the equivalence
relation generated by the following operations.
\begin{enumerate}
\item Operation.  multiplication by
a unimodular complex number $a$ of any row and by $\overline{a}$ of the
corresponding column.
\item Operation. Interchange of rows and simultaneously
of the corresponding columns.
\end{enumerate}
The purpose of this paper
is to investigate Seidel matrices of order $2k$ which have two distinct
eigenvalues with equal multiplicity $k$. This class of matrices is a
particular subclass of the class of \ the so-called complex conference
matrices. A complex $n\times n$ \textit{conference matrix} $C_{n}$ is a matrix
with $c_{ii}=0$ and $\left\vert c_{ij}\right\vert =1$, $i\neq j$ that
satisfies

\[
CC^{\ast}=(n-1)I_{n}. %
\]

Real conference matrices have been heavily studied in the literature in
connection with combinatorial designs in geometry, engineering, statistics,
and algebra.

The following necessary conditions are known : $n\equiv2$ $(\operatorname{mod}%
4)$ and $n-1=a^{2}+b^{2}$, $a$ and $b$ integers for symmetric matrices
(\cite{B}, \cite{P}, \cite{R}, \cite{LiS}), and $n=2$ or $n\equiv0$ $(\operatorname{mod}%
4)$ for skew symmetric matrices \cite{W}. However the only conference
matrices that have been constructed so far are symmetric matrices of order
$n=p^{\alpha}+1\equiv2$ $(\operatorname{mod}4)$, $p$ prime, $\alpha$
non-negative integer (\cite{P}) or $n=(q-1)^{2}+1$, where $q$ is the order of
a conference symmetric or skew symmetric matrix (\cite{GS}) or $n=(q+2)q^{2}%
+1,$ where $q=4t-1=p^{\alpha}$, $p$ prime and $q+3$ is the order of a
conference symmetric matrix (\cite{M}), or$\ n=5\cdot9^{2\alpha+1}+1$, $\alpha$ non-negative integer
(\cite{SZ}), and skew symmetric matrices of order $n=2^{s}\prod\limits_{i=1}%
^{r}(p_{i}^{\alpha_{i}}+1)$, $p_{i}^{\alpha_{i}}+1\equiv0(\operatorname{mod}%
4)$, $p_{i}$ primes, $s$, $r$ and $\alpha_{i}$ non-negative integers
(\cite{W}). In \cite{DGS} the authors show that essentially there are no other
real conference matrices. Precisely they prove that any real conference matrix
of order $n>2$ is equivalent, under multiplication of rows and columns by
$-1$, to a conference symmetric or to a skew symmetric matrix according as $n$
satisfies $n\equiv2$ $(\operatorname{mod}4)$ or $n\equiv0$
$(\operatorname{mod}4)$. In addition we observe that $n$ must be even. This is
not the case for complex conference matrices. They were used in \cite{D} to provide
parametrization of complex Sylvester inverse orthogonal matrices. It is quoted
in \cite{D}, and easy to show, that there is no complex conference matrix of
order $3$; however we can find such a matrix of order $5$ \cite{D}. Recently, the present author \cite{pr} constructed an infinite family of complex symmetric conference matrices of odd orders by use of finite fields and the Legendre symbol. These matrices solve the problem of finding the maximum number of pairwise isoclinic planes in Euclidean odd dimensional spaces. In this article, a new method to obtain a parametrization of complex conference matrices of even orders is given. Complex conference matrices are important because by construction the matrix
\[
H_{2n}=\left(
\begin{array}
[c]{cc}%
C_{n}+I_{n} & C_{n}^{\ast}-I_{n}\\
C_{n}-I_{n} & -C_{n}^{\ast}-I_{n}%
\end{array}
\right)  ,
\]

is a complex Hadamard matrix of order $2n$.

A matrix of order $n$ with unimodular entries and satisfying $HH^{\ast}%
=nI_{n}$ is called \textit{complex Hadamard}.

The aim of this article is to relate directly the three objects introduced one
to each other and to derive a connection between complex conference matrices ($CCMs$), complex Hadamard matrices
($CHMs$) and a class of complex
equiangular tight frames ($CETFs$). The correspondance here developed turns
out to be really fruitful : on the one hand we construct new, previously
unknown parametric families of complex conference matrices and of complex
Hadamard matrices of square orders, on the other hand we show amongst other
that for any integer $k$ such that $2k=p^{\alpha}+1\equiv2$
$(\operatorname{mod}4)$, $p$ prime, $\alpha$ non-negative integer, there is a family of
$(2k,k)$ complex equiangular tight frames which depends on a complex number of
modulus $1$, and for any $\beta\in\mathbb{N}^{\ast}$, there exists a family of
$((4k)^{2^{\beta}},\frac{1}{2}(4k)^{2^{\beta-1}}((4k)^{2^{\beta-1}}\pm1))$
complex equiangular tight frames which depends on two complex numbers of
modulus one. For instance we have the existence of a family of $(2k,k)$
$CETFs$ coming from fourth root Seidel matrices. This yields amongst other the
existence of $(6,3)$, $(10,5)$, $(14,7)$, $(18,9)$ $CETFs$ which have been
found in \cite{DHS}.
First let us survey the present knowledge on $(2k,k)$ equiangular tight frames.
\begin{enumerate}
\item Any real conference symmetric matrix $C$ of order $2k$ is a Seidel matrix with two eigenvalues and leads to
a $(2k,k)$ $RETF$. In the same time the matrix $iC$ is a complex symmetric conference matrix of order $2k$.
\item Any real conference skew symmetric matrix $C$ of order $2k$ leads to a
Seidel matrix ($iC)$ with two eigenvalues and then to a $(2k,k)$ $CETF$. Note that the $2k$ vectors of this frame
generate a set of equiangular lines called in \cite{Id0} an $F$%
\textit{-regular }$2k$\textit{-tuple} in $\mathbb{CP}^{k-1}$. This is a tuple
in which all triples of lines are pairwise congruent. Related objects can also
be found in \cite{Id}. See also $A146890$ in the On-line
Encyclopedia of Integer sequences.
\item
Zauner constructed in his thesis (\cite {Z}) $(q+1,(q+1)/2)$ $CETFs$ for any odd prime power $q$. However it is shown in section $5$ that the associated Seidel matrices of these frames are real symmetric conference matrices or the product by $i$ of real skew symmetric conference matrices. That is Zauner's construction does not lead to new $(2k,k)$ $CETFs$ in comparison with examples 1 and 2. In the same time it is interesting to see how Zauner obtained again the Paley matrices using an additive character on $GF(q)$  the Galois field of order $q$.
\item
In \cite{DHS} the authors studied the existence and construction of Seidel matrices with two eigenvalues whose off diagonal entries are all fourth roots of unity. Among other things they constructed using an algorithm some $CETFs$ which do not arise from real skew symmetric matrices. Their associated Seidel matrices with fourth roots of unity can be obtained directly from  our one parametric family of $(2k,k)$ when the parameter is equal to $i$. Their results are then here improved.
\end{enumerate}
Of course example 2 is well-known but a continuous family of complex equiangular tight frames appears in this paper for the first time.
\subsection{Seidel matrices and CHMs}

Let $H$ be a complex Hadamard matrix of order $n$ and let us denote its rows
by $h_{1}$, $h_{2}$,...,$h_{n}$.

Consider the following block matrix, where the $(i,j)$ th entry of $K$ is the
block $h_{j}^{\ast}h_{i}$,%
\begin{equation}
K=\left(
\begin{array}
[c]{ccc}%
h_{1}^{\ast}h_{1} & ... & h_{n}^{\ast}h_{1}\\
... & ... & ...\\
h_{1}^{\ast}h_{n} & ... & h_{n}^{\ast}h_{n}%
\end{array}
\right)  .
\end{equation}
$\ $ It is well known that $K$ is an Hermitean complex Hadamard matrix of
order $n^{2}$ with constant diagonal $1$ \cite{So2}. Clearly we have the
useful following Lemma.

\begin{lemma}
$Q:=K-I$ is a Seidel matrix associated to a $(n^{2},\frac{n(n+1)}{2})$ $CETFs$.
\end{lemma}

In \cite{BE} the authors applied this block constuction to the Fourier matrix
$H=\left[  e^{\frac{2\pi i(k-1)(l-1)}{n}}\right]  _{k,l}$ of order $n$, from
which they concluded that for any $n\geq2$ there exists an equiangular
$(n^{2},\frac{n(n+1)}{2})$ frame. Butson constructed in $\cite{BUT}$ for any
prime $p$, complex Hadamard matrices of order $2^{a}p^{b},0\leq a\leq b$,
containing $pth$ roots of unity. In \cite{So2} the block construction $(1)$
applied to Butson matrices yields the existence of an equiangular
$(4^{a}p^{2b},2^{a-1}p^{b}(2^{a}p^{b}+1))$ frame\textit{. }For instance there
is a $(36,21)$ C\textit{ETF }arising from Seidel\textit{\ }matrices containing
cube roots of unity.

\section{Constructing complex conference matrices from real symmetric conference
matrices}

In the following we will restrict ourselves to real symmetric conference matrices called
Paley matrices : those of orders $p^{\alpha}+1\equiv2$ $(\operatorname{mod}4)$. It is shown in \cite{GS}
that any Paley matrix of order $n=2k=p^{\alpha}+1\equiv2$ $(\operatorname{mod}%
4)$ is equivalent to a
matrix of the form%

\[
C=\left(
\begin{array}
[c]{cccc}%
0 & 1 & j^{T} & j^{T}\\
1 & 0 & -j^{T} & j^{T}\\
j & -j & A & B\\
j & j & B^{T} & -A
\end{array}
\right)  ,
\]

where $A$ and $B$ are square matrices of order $k-1$, $j$ is the
$(k-1)\times1$ matrix consisting solely of $1^{\prime}s$. Because this form of $C$ slightly differs from the form of \cite{GS} we
give here the proof for completeness of this paper.

Let $V$ \ be a vector space of dimension $2$ over $GF(p^{\alpha})$ the Galois
field of order $p^{\alpha}$, $p$ prime, $\alpha$ non-negative integer. Let
$\varkappa$ denote the Legendre symbol, defined by $\varkappa(0)=0$,
$\varkappa(a)=1$ or $-1$ according as $a$ is or not a square in $GF(p^{\alpha
})$. Then clearly $\varkappa(-1)=1$. Now let $x$, $y$ be two independent vectors
in $V$, $\alpha_{i}$ being the elements of $GF(p^{\alpha})$ and consider
$\varkappa\det$ where $\det$ is any alternating bilinear form on $V$. The
Paley matrix $C$ is defined as follows : $C=\left[  \varkappa\det(y_{i}%
,y_{j})\right]  _{i,j=0,...,p^{\alpha}}$ where $y_{0}=x$, $y_{j}=y+\alpha
_{j}x$ for all $j=1,...,p^{\alpha}$. The result follows by taking
$\varkappa\det(y,x)=1$ and by arranging the vectors as follows : $x$, $y$,
$y+x\eta$, $y+x\eta^{3}$,..., $y+x\eta^{p^{\alpha}-2}$, $y+x\eta^{2}$,
$y+x\eta^{4}$,..., $y+x\eta^{p^{\alpha}-1}$ where $\eta$ denotes any primitive
element of $GF(p^{\alpha})^{\ast}$.

The matrices $A$ and $B$ satisfy :
\begin{align}
^{t}A  &  =A,AJ=J,BJ=JB=0,\\
AB  &  =BA,BB^{T}=B^{T}B,\\
A^{2}+BB^{T}  &  =(2k-1)I-2J,
\end{align}
and $\ J$ \ is the matrix of order $k-1$ which entries are $1$'s. These
equations were not considered in \cite{GS} since their goal was different from ours.

\begin{theorem}
Let $k\geq 3$ be an odd integer such that $2k=p^{\alpha}+1$. There exists an infinite family of complex conference matrices of order $2k$
depending on two complex parameters $a$ and $b$ of modulus $1$.
\end{theorem}

\begin{proof}
Consider the matrix%
\[
C(a,b)=\left(
\begin{array}
[c]{cccc}%
0 & 1 & j^{T} & j^{T}\\
1 & 0 & -aj^{T} & aj^{T}\\
j & -aj & aA & abB\\
j & aj & a\overline{b}B^{T} & -aA
\end{array}
\right)  \text{.}%
\]
A direct calculation with equations $(2)$, $(3)$ and $(4)$ leads to
$C(a,b)C^{\ast}(a,b)=(2k-1)I.$ That is $C(a,b)$ is a complex conference matrix
of order $2k$.
\end{proof}

The matrix $C(a,b)$ of order $6$ appears already in \cite{D}. All the other
cases are new in comparison with that paper.

\begin{corollary}
Let $k\geq 3$ be an odd integer such that $2k=p^{\alpha}+1$. There exists a\ parametric family of complex Hadamard matrices of order $4k$.
\end{corollary}

\begin{proof}
The matrix $H(a,b)=\left(
\begin{array}
[c]{cc}%
C(a,b)+I & C^{\ast}(a,b)-I\\
C(a,b)-I & -C^{\ast}(a,b)-I
\end{array}
\right)  $ yields the solution.
\end{proof}

\begin{theorem}
Let $k\geq 3$ be an odd integer such that $2k=p^{\alpha}+1$. There exists an infinite family of Hermitean complex Hadamard matrices of
order $16k^{2}$ depending on two complex parameters $a$ and $b$ of modulus $1
$ with a constant diagonal of $1^{\prime}s$.
\end{theorem}

\begin{proof}
It suffices to apply the block construction $(1)$ to $H(a,b)$.
\end{proof}

Applying Lemma 1 to $K(a,b)$ and repeating the same Lemma to the obtained
matrix and repeating this operation indefinitely implies the following.

\begin{theorem}
For any odd integer $k\geq 3$ such that $2k=p^{\alpha}+1$ and for any $\beta\in\mathbf{N}^{\ast}$ there exists a $((4k)^{2^{\beta}}%
,\frac{(4k)^{2^{\beta-1}}((4k)^{2^{\beta-1}}\pm1)}{2})$ $CETF$.
\end{theorem}

$\pm$ correpond to the frame and its conjugate.

Interestingly, the redundancy of this frame is the general term of a sequence
of $\beta$ which converges to $2.$

\begin{theorem}
For any integer $k\geq3$ such that $2k=p^{\alpha}+1\equiv2$ $(\operatorname{mod}4)$ there
exists a $(2k,k)$ $CETF$.

\begin{proof}
$C(1,b)$ is a complex Hermitean conference matrix wich leads to $(2k,k)$
$CETFs$.
\end{proof}
\end{theorem}

As above repeating the construction $(1)$ indefinitely yields the following.

\begin{corollary}
For any integer $k\geq3$ such that $2k=p^{\alpha}+1\equiv2$ $(\operatorname{mod}4)$ and for
any $\beta\in\mathbf{N}^{\ast}$ there exists a $((2k)^{2^{\beta}},\frac{(2k)^{2^{\beta-1}}((2k)^{2^{\beta-1}}\pm1)}{2})$
$CETF$.
\end{corollary}

Clearly we come up to the following proposition about the redundancy of these frames.

\begin{proposition}
The redundancy of the $((2k)^{2^{\beta}},\frac{(2k)^{2^{\beta-1}%
}((2k)^{2^{\beta-1}}\pm1)}{2})$ CETF is the general term of a sequence of
$\beta$ which converges to $2$.
\end{proposition}

From Theorem $6$ we deduce for instance the existence of
$\ (6,3),(10,5),(14,7),$ $(18,9),(26,13)....CETFs$. It turns out that for any
odd integer $k$ smaller than $50$ we may construct a $(2k,k)$ $CETF$ except
possibly in the cases $k=11,17,23,29,33,35,39,43,47$.

If we put $b=\pm i$ then $H=C(1,b)\pm iI$ \ is a \textit{quaternary} complex
Hadamard matrix of order $2k$, that is its entries are fourth roots of unity (\cite{LSO}).
The block construction $(1)$ yields a quaternary Hadamard matrix of order
$4k^{2}$.

\section{Constructing complex Hermitean conference matrices from real
symmetric conference matrices}

The results of this section are a corollary of a theorem presented in $\cite{So}$, which provides parametrizations of complex
Hadamard matrices.

\begin{proposition}
If $D_{n}(t)$ is a parametric family of complex Hadamard matrices, coming from the " conference matrix construction " i.e. $D_{n}(t)=I+iC$ where $C$ is a real symmetric conference matrix of order $n$, then $C(t)=-i(D_{n}(t)-I)$ is a family of complex Hermitean conference
matrices.
\end{proposition}

It appears that the construction presented in $\cite{So}$ and improved in \cite{LSO} also works for complex Hermitean conference matrices. We may introduce more parameters.
In the following we used that
construction to obtain the families stemming from the real conference matrices
$C_{6}$, $C_{10}$ and $C_{14}$ which are unique up to equivalence. It appears
that we may introduce one parameter in $C_{6}$, two parameters in $C_{10}$ and
six parameters in $C_{14}$.
\[
C_{6}(b)=\left(
\begin{array}
[c]{cccccc}%
0 & 1 & 1 & 1 & 1 & 1\\
1 & 0 & -1 & b & 1 & -b\\
1 & -1 & 0 & -b & 1 & b\\
1 & \overline{b} & -\overline{b} & 0 & -1 & 1\\
1 & 1 & 1 & -1 & 0 & -1\\
1 & -\overline{b} & \overline{b} & 1 & -1 & 0
\end{array}
\right)  .
\]

To construct the matrix $C_{6}(b)$ we consider the pair of rows $(2,3)$. If we
choice another suitable pair of rows we do not obtain the same equivalence
class, for instance if we choice the pair $(2,4)$ and interchange the second
and third rows and the corresponding columns we obtain the matrix $C_{6}(1,b)$
of Theorem $6$ which is permutation equivalent to $C_{6}(\overline{b})$.
Although $C_{6}(\overline{b})$ is the conjugate of $C_{6}(b)$ they are not
permutation equivalent, except in cases $b\in\{\pm1,\pm i\}$. This means that
the two matrices lead to two non congruent $(6,3)$ $CETFs$, whereas in
$\mathbb{CP}^{2}$ the two $6$-tuples of lines generated by the frames are
congruent \cite{BET}. If we consider now two unimodular complex numbers $b,b^{\prime
}\notin$ $\{\pm1,\pm i\}$ such that $b^{\prime}\neq b$ and $b^{\prime}\neq-b$
then the matrices $C_{6}(b)$ and $C_{6}(b^{\prime})$ are non equal normal
forms and they are not permutation equivalent. This means that we have an
infinite family depending on one parameter of non equivalent complex Hermitean
conference matrices of order $6$, or equivalently we have an infinite family
depending on one parameter of non congruent $(6,3)$ $CETFs$.%

\[
C_{10}(a,b,c)=\left(
\begin{array}
[c]{cccccccccc}%
0 & 1 & 1 & 1 & 1 & 1 & 1 & 1 & 1 & 1\\
1 & 0 & a\overline{b} & a & \overline{c} & 1 & -\overline{c} & -a &
-a\overline{b} & -1\\
1 & \overline{a}b & 0 & -b & -\overline{c} & 1 & \overline{c} & b & -1 &
-\overline{a}b\\
1 & \overline{a} & -\overline{b} & 0 & 1 & -1 & 1 & -1 & \overline{b} &
-\overline{a}\\
1 & c & -c & 1 & 0 & -1 & -1 & 1 & -c & c\\
1 & 1 & 1 & -1 & -1 & 0 & -1 & -1 & 1 & 1\\
1 & -c & c & 1 & -1 & -1 & 0 & 1 & c & -c\\
1 & -\overline{a} & \overline{b} & -1 & 1 & -1 & 1 & 0 & -\overline{b} &
\overline{a}\\
1 & -\overline{a}b & -1 & b & -\overline{c} & 1 & \overline{c} & -b & 0 &
\overline{a}b\\
1 & -1 & -a\overline{b} & -a & \overline{c} & 1 & -\overline{c} & a &
a\overline{b} & 0
\end{array}
\right)  .
\]
As in $\cite{So}$ the considered pairs of rows are $(2,10)$, $(3,9)$ and
$(5,7)$.%

\newpage
\[
C_{14}(a,b,c,d,e,f)=
\]%
\[
\left(
\begin{array}
[c]{cccccccccccccc}%
0 & 1 & 1 & 1 & 1 & 1 & 1 & 1 & 1 & 1 & 1 & 1 & 1 & 1\\
1 & 0 & -1 & a\overline{b} & -a\overline{b} & -\overline{c} & a & \overline{e}
& \overline{c} & 1 & -a & -\overline{e} & a & -a\\
1 & -1 & 0 & -a\overline{b} & a\overline{b} & -\overline{c} & -a &
\overline{e} & \overline{c} & 1 & a & -\overline{e} & -a & a\\
1 & \overline{a}b & -\overline{a}b & 0 & -1 & b & -b & -\overline{e} & b & 1 &
\overline{f} & \overline{e} & -b & -\overline{f}\\
1 & -\overline{a}b & \overline{a}b & -1 & 0 & -b & b & -\overline{e} & -b &
1 & \overline{f} & \overline{e} & b & -\overline{f}\\
1 & -c & -c & \overline{b} & -\overline{b} & 0 & -c\overline{d} & c & -1 &
-1 & 1 & c & c\overline{d} & 1\\
1 & \overline{a} & -\overline{a} & -\overline{b} & \overline{b} &
-d\overline{c} & 0 & -d\overline{e} & d\overline{c} & -1 & -d\overline{f} &
d\overline{e} & 1 & d\overline{f}\\
1 & e & e & -e & -e & \overline{c} & -e\overline{d} & 0 & -\overline{c} & 1 &
-\overline{f} & -1 & e\overline{d} & \overline{f}\\
1 & c & c & \overline{b} & -\overline{b} & -1 & c\overline{d} & -c & 0 & -1 &
1 & -c & -c\overline{d} & 1\\
1 & 1 & 1 & 1 & 1 & -1 & -1 & 1 & -1 & 0 & -1 & 1 & -1 & -1\\
1 & -\overline{a} & \overline{a} & f & f & 1 & -f\overline{d} & -f & 1 & -1 &
0 & -f & f\overline{d} & -1\\
1 & -e & -e & e & e & \overline{c} & e\overline{d} & -1 & -\overline{c} & 1 &
-\overline{f} & 0 & -e\overline{d} & \overline{f}\\
1 & \overline{a} & -\overline{a} & -\overline{b} & \overline{b} &
d\overline{c} & 1 & d\overline{e} & -d\overline{c} & -1 & d\overline{f} &
-d\overline{e} & 0 & -d\overline{f}\\
1 & -\overline{a} & \overline{a} & -f & -f & 1 & f\overline{d} & f & 1 & -1 &
-1 & f & -f\overline{d} & 0
\end{array}
\right)  .
\]
The considered pairs of rows are : $(2,3)$, $(4,5)$, $(6,9)$, $(7,13)$,
$(8,12)$ and $(11,14)$.

Now relating these matrices to frames we see for instance that we have
$(14,7)$ $CETFs$ depending on six unimodular complex numbers.

 It would be interesting to see whether the hereby presented
constructions can be applied to real conference matrices which are not of the
Paley type. Our complex conference matrices can be used efficiently to obtain
previously unknown real or complex Hadamard matrices and $CETFs$. These
matrices can be also used as starting points to construct Sylvester inverse
orthogonal matrices by the method developed in \cite{D}.\bigskip\

\section{Zauner's construction}

We show in the following that Seidel matrices associated to Zauner's frames are real symmetric conference matrices or the product by $i$ of real skew symmetric conference matrices. First we recall the construction. For any odd prime power $q=p^{m}$ let $GF(q)$ be the Galois field of order $q$, $\varkappa$ be the legendre symbol which is a multiplicative character of $GF(q)^{*}$ and let $\psi$ be the additive character defined by $\psi(a)=e^{2i\pi Tr(a)/p}$ where the $Tr$ is the linear mapping from $GF(q)$ to $F_{p}$ such that $Tr(a)=a^{p}+...+a^{p^{m}}$.  Now let $a_{1},...,a_{q}$ be the elements of $GF(q)$, $b_{1},...,b_{(q-1)/2}$ be the non-zero squares and $b'_{1},...,b'_{(q-1)/2}$ be the non-zero non squares.
The following vectors are given in \cite{Z}.
\begin{center}
$x_{1}=(1,0,...,0)$,
\end{center}
\begin{center}
$x_{2}=(1/\sqrt{q},\sqrt{2/q}\psi(b_{1}a_{1}),...,\sqrt{2/q}\psi(b_{(q-1)/2}a_{1})$,
...,
\end{center}
\begin{center}$
x_{q+1}=(1/\sqrt{q},\sqrt{2/q}\psi(b_{1}a_{q}),...,\sqrt{2/q}\psi(b_{(q-1)/2}a_{q}))$.
\end{center}
Based on a formula on additive characters the author showed in \cite{Z} that his vectors are unit and that the absolute value of any hermitean product $<x_{k},x_{l}>$ with $k\neq{l}$, is equal to $\frac{1}{\sqrt{q}}$. In the following the hermitean products $<x_{k},x_{l}>$ with $k\neq{l}$ are computed.
For any $2\leq k < l\leq q+1$ the hermitean product $<x_{k},x_{l}>$ is equal to
\begin{center}
$\frac{1}{q}+\frac{2}{q}\sum\limits_{s=1}^{\frac{q-1}{2}}\psi(b_{s}(a_{k}-a_{l}))$.
\end{center}
On the one hand
\begin{center}
$\sum\limits_{s=1}^{\frac{q-1}{2}}\psi(b_{s}(a_{k}-a_{l}))+\sum\limits_{s=1}^{\frac{q-1}{2}}\psi(b'_{s}(a_{k}-a_{l}))=\sum\limits_{\alpha \in GF(q)}^{}\psi(\alpha)-1=-1$,
\end{center}
because $\sum\limits_{\alpha \in GF(q)}^{}\psi(\alpha)=0$.
On the other hand
\begin{center}
$\sum\limits_{s=1}^{\frac{q-1}{2}}\psi(b_{s}(a_{k}-a_{l}))-\sum\limits_{s=1}^{\frac{q-1}{2}}\psi(b'_{s}(a_{k}-a_{l}))=\frac{1}{\varkappa(a_{k}-a_{l})}\sum\limits_{\alpha \in GF(q)}^{}\varkappa(\alpha)\psi(\alpha)$.
\end{center}
However the last sum is a general Gauss sum which was computed in \cite{BEV}. It turns out that this sum is equal to
\begin{center}
$(-1)^{m-1}\sqrt{q}$ if $p\equiv1(\operatorname{mod}4)$ and
\end{center}
\begin{center}
$-(-i)^{m}\sqrt{q}$ if $p\equiv-1(\operatorname{mod}4)$.
\end{center}
From this follows clearly that
\begin{center}
$<x_{k},x_{l}>=(-1)^{m-1}\frac{1}{\sqrt{q}}\varkappa(a_{k}-a_{l})$ if $p\equiv1(\operatorname{mod}4)$ and
\end{center}
\begin{center}
$<x_{k},x_{l}>=-(-i)^{m}\frac{1}{\sqrt{q}}\varkappa(a_{k}-a_{l})$ if $p\equiv-1(\operatorname{mod}4)$.
\end{center}
We see from the Gram matrices that with this construction we find again Paley matrices and no other complex conference matrices.

\end{document}